\documentclass[12pt]{article}

\usepackage{graphicx} % Extended graphics package.
\usepackage{amsmath} % American Mathematics Society standards
\usepackage{amsxtra} % Additional math symbols
\usepackage{amssymb} % Additional math symbols
\usepackage{amsthm} % Additional math symbols
\usepackage{latexsym} % Additional math symbols
\usepackage{setspace} % Controls line spacing
\usepackage{bbm}      % for the indicator function
\usepackage{multicol}
\usepackage{dirtytalk}
\usepackage[ruled,vlined]{algorithm2e}

\usepackage{pgf,tikz,pgfplots}
\pgfplotsset{compat=1.14}
\usepackage{mathrsfs}
\usetikzlibrary{arrows}
\usepackage{graphics}
\usepackage{tikz}
\usepackage{wrapfig}
\usepackage{caption}
\usepackage{subcaption}
\usepackage[margin=1in]{geometry} % Sets page margins to 1 inch on all sides
\usepackage[titles]{tocloft} % Adds leader dots to all entries in the table of contents
\usepackage[ruled,vlined]{algorithm2e}
\include{pythonlisting}

\newcommand{\minimize}{\text{minimize}}
\newcommand{\sub}{\text{subject to}}

\DeclareMathOperator{\Mod}{Mod}
\DeclareMathOperator{\Adm}{Adm}
\DeclareMathOperator{\gen}{gen}

\newtheorem{theorem}{Theorem}[section]
\newtheorem{definition}{Definition}[section]
\newtheorem{lemma}{Lemma}[section]

\newtheorem{corollary}{Corollary}[section]

\title{$p$-Modulus on radially symmetric trees}
  \author{Prem Raj Prasain \\
    \small Department of Physical and Computational Science \\ \small Bethany College, Bethany, WV }

\begin{document}

\maketitle

\begin{abstract}

In this paper, we establish the theory of $p$-modulus of a family of infinite paths on an infinite-rooted tree and then explore its interpretation and properties. One key result is the formulation of $p$-modulus on the infinite tree as a limit of $p$-modulus on truncated trees, with a formula given in terms of a series. Analogous to the existing theory for finite graphs, the $1$-modulus of a family of descending paths in an infinite tree is related to the minimum cut problem, the $2$-modulus is related to effective resistance, and the $\infty$-modulus is related to the length of shortest paths. Another key result is the existence of a critical $p$-value for radially symmetric infinite binary trees, which assigns a kind of dimension to the boundaries of these trees.
%and it draw line between positive and zero value of $p$-modulus. At the critical $p$-value, there no general theory to determine the modulus is either positive or zero. The positive (zero) value of modulus implies an infinite tree is bushy (or skinny) or a random walk is transient (recurrent).
\end{abstract}
\textbf{Keywords:} graph, modulus, radially symmetric infinite tree, descending paths, truncated tree. 

\section{Introduction}\label{Introduction}
A \emph{graph} is a tuple $G=(V,E)$ consisting of the set of \emph{vertices} $V$ and the set of \emph{edges} $E$.  The number of vertices and edges in the given graph $G=(V,E)$ are represented by $|E|$ and $|V|$, respectively. A graph $G$ is \emph{finite} if both $|E|$ and $|V|$ are finite, otherwise $G$ is \emph{infinite}. A \emph{sequence} $\gamma$ in a graph $G$ is the chain of vertices and edges. That is, $\gamma=v_0e_1v_1e_2v_2\cdots $ $v_i \in V$, $i= 1, 2, 3, \cdots$ and $e_i=\{v_{i-1},v_i\}\in E$, $i=1, 2, 3, \cdots$. A \emph{finite walk} in a graph $G$ is a finite sequence $\gamma=v_0e_1v_1e_2v_2\cdots e_nv_n$. In this case, the walk is said to have \emph{hop length} $n$, and it is denoted by $\ell(\gamma)$, that is, $\ell(\gamma)=n$. In a similar fashion, an \emph{infinite walk} is an infinite sequence $\gamma=v_0e_1v_1e_2v_2 \cdots$ where $e_i \neq e_j$ and in this case, we say that $\gamma$ has length $\ell(\gamma)=+\infty$. A \emph{finite (resp.\  infinite) path} is a finite (resp.\ infinite) walk in which $v_i\ne v_j$ for all $i\ne j$.  A graph is called \emph{simple} if it contains at most one edge between any two vertices and contains no self-loops (i.e., edges of the form $\{v,v\}$). In a simple graph, a walk can be uniquely determined from either its vertex sequence $v_0v_1v_2\cdots$ or its edge sequence $e_1e_2e_3\cdots$.  A \emph{cycle} is a finite walk $v_0e_1v_1e_2\cdots v_{n-1}e_nv_0$ with the property that $v_0e_1v_1e_2\cdots v_{n-1}$ is a path. 

Let $\Gamma(u,v)$ be the set of all connecting paths between given two vertices $u$ and $v$ in the graph $G$. A graph $G$ is \emph{connected} we mean the set $\Gamma(u,v)\ne\emptyset$ for all distinct pairs of vertices $u$ and $v$; that is, the graph $G$ has at least one vertex, and there exists a path between every pair of vertices. A \emph{weighted graph} is a tuple $G=(V,E, \sigma)$ where the function $\sigma: E \to \mathbb{R}_{> 0}$ which assigns a positive weight to each edge of $G$. The case $\sigma\equiv 1$ can be thought of as an unweighted graph and denoted by $G=(V,E)$.  So, without loss of generality, we shall consider all graphs to be weighted. \\
A graph $G$ is a \emph{tree} if it is connected and contains no cycles. A \emph{rooted tree} is a tuple $G=(G,E,\sigma,o)$ where $G=(V,E,\sigma)$ is a tree and $o\in V$ is a specially identified vertex called the \emph{root}. Note that for given any two vertices $u$ and $v$ in a tree, $\Gamma(u,v)$ contains exactly one simple path $\gamma_{uv}$. Every vertex, $v$, in a rooted tree has a \emph{generation}, $\gen(v)$ defined as the number of ``hops'' from the root $o$ to vertex $v$; that is, 
\begin{equation*}
\gen(v) = 
\begin{cases}
0 &\text{if }v = o,\\
\ell(\gamma_{ov}) &\text{if }v\ne o.
\end{cases}
\end{equation*}
Given two distinct edges $e,e'\in E$ and $e$ and $e'$ are adjacent, then $e$ is called the \emph{parent} of $e'$ (or $e'$ is called a \emph{child} of $e$) and $e$ is an \emph{ancestor} of $e'$ ( or $e'$ is a \emph{descendent} of $e$) if every path from the root that contains $e'$ also contains $e$. Equivalently, given two vertices $u$ and $v \in V$, we say $u$ is a \emph{parent} of $v$ (or $v$ is a \emph{child} of $u$) if $\{u,v\} \in E$ and $u$ is an \emph{ancestor} of $v$ (or $v$ is a descendent of $u$) if the path $\gamma_{ov}$ passes through the vertex $u$. Note that in a rooted tree, each edge other than the edges of the first generation has a unique parent, but the edge can have many children. Similarly, each vertex other than the root has a unique parent, but the vertex can have many children. Given an edge $e \in E$, the unique parent edge is denoted by $p(e)$, and the set of children is denoted by $c(e)$. The number of children of edge $e$ is indicated by $C(e)=|c(e)|$. We call a tree locally finite if $C(e)<\infty$ for all $e\in E$. In this paper, a tree  is a \emph{proper infinite tree} if it is infinite, locally finite and $C(e)\ge 1$ for all $e\in E$.\\

A \emph{radially symmetric tree} is a proper tree for which the number of generation $C(e)$ of an edge $e$ depends only on the generation of $e$; that is, $C(e)=C(gen(e))$. In other words, an infinite tree in which each edge of a generation has the same number of children. By the \emph{ball $B_n$ of radius n}, we mean the set of edges that are within $n$ hops from the root. That is, \begin{equation*}
    B_n :=\{e \in E: \gen(e) \le n \}
\end{equation*}
By the \emph{shell of radius n}, we mean the set
\begin{equation*}
    S_n:=B_n \setminus B_{n-1}=\{e \in E: \gen(e)=n\}
\end{equation*}
Let $G=(V,E,o)$ be a proper tree and let $G'=(V',E',o)$ be a \emph{subtree} with same root $o$. We say that $G'$ is a truncation of tree $G$ if 
\begin{equation*}
    \forall e \in E \quad \text{either} \quad c(e) \cap E' = \emptyset \quad \text{or} \quad  c(e) \cap E' = c(e)
\end{equation*}
In words, for each $e \in E$, $G'$ either contains all or none of $e'$s children in $G$.\\

In this paper, in a rooted tree, every path $\gamma=e_1e_2e_3\cdots$ starting at the root, $o$ is considered as a \emph{descending path} in which $e_i=p(e_{i+1})$ for $i=1,2,3, \cdots$. The family of all descending paths in the proper infinite tree $G$ is represented by $\Gamma$. The \emph{length} of a descending path $\gamma\in \Gamma$ can be defined as   

\begin{equation*}
\ell(\gamma) := \sup_{e\in\gamma}\gen(e),
\end{equation*}
If $\gamma$ is finite, then the $\ell(\gamma)$ has some finite value and coincides graph length of $\gamma$. Otherwise, we say it has a length of $+\infty$. For $n \ge 1$, we define subfamilies $\Gamma_n$ and $\Gamma_{\infty}$ of the family $\Gamma$ as   
\begin{equation*}
	\Gamma_n := \{\gamma\in\Gamma : \ell(\gamma) = n\} \quad\text{and}\quad
	\Gamma_\infty := \{\gamma\in\Gamma : \ell(\gamma) = +\infty \}
\end{equation*}
respectively. For a given $1\le m\le n\le\infty$ and if $\gamma\in\Gamma_n$ then there exists a unique $\gamma'\in\Gamma_m$ such that $\gamma'$ is a sub-path of $\gamma$. This sub-path relationship is denoted by $\gamma'\preceq\gamma$. A \emph{cut} set for $\Gamma$ in an infinite tree is a subset of edges $C\subseteq E$ such that $|\gamma\cap C|>0$ for every $\gamma\in\Gamma$. \newline

Initially, the theory of \emph{$p$-}modulus came out of the theory of conformal modulus in complex analysis (see \cite{Ahlfors1973CONFORMALTHEORY}). Intuitively, \emph{$p$-}modulus provides a method for quantifying the richness of a family of curves in the sense that a family with many short curves will have a larger modulus than a family with fewer and longer curves. The \emph{$p$-}modulus in the discrete finite graph has been studied in~\cite{Resume2009EMPILEMENTSCOMBINATOIRES, Odedmodulusindiscrete}. Recently, the \emph{$p$-}modulus of a family of objects on the finite networks has been introduced, and its general properties have been studied (see, ~\cite{Albin2015ModulusQuantities, Albin2018BlockingApplications, AlbinModulusGraphs, Albin2018ModulusMO}). The \emph{$p$-}modulus in networks has been a versatile tool to measure structural properties and applications, including clustering and community detection, the measure of centrality, the construction of a large class of graph metrics, hierarchical graph decomposition, and the solution to game-theoretic models of
secure network broadcast (see \cite{ALBIN2021112282, securebroadcast, SHAKERI2018127, PhysRevE.95.012316}).  \\

An infinite undirected tree can be thought of as an electrical network with unit resistance (see \cite{Doyle2000RandomNetworks, Vatamanelu2010TheNetworkb}). The effective resistance in an infinite rooted tree is the reciprocal of the limit of the effective conductance in a truncated tree. (see \cite{ Dueein1962TheNetwork, LyonsProbabilityNetworks}). The modulus $\Mod_{2, \sigma }(\Gamma_n)$ of the family of descending paths from the root on the truncated (finite) rooted tree is the effective conductance (the reciprocal of the effective resistance) (see \cite{AlbinModulusGraphs}). A random walk on an infinite connected tree is related to the effective conductance of the tree. The random walk is transient if and only if the effective conductance from any of its vertices to infinity is positive (see, \cite{Doyle2000RandomNetworks, LyonsProbabilityNetworks}). A random walk's transient or recurrent nature can be expressed in terms of the voltage or the escape probability. (see,\cite{Vatamanelu2010TheNetworkb, LyonsProbabilityNetworks}).

In this paper, we expand upon the finite case to develop a theoretical framework for the \emph{$p$-}modulus on proper infinite trees and extend several fundamental properties of the modulus on finite graphs to such trees. We prove that the $p$-modulus of the family of infinite descending paths on an infinite proper tree is the limit of the \emph{$p$-}modulus of a family of finite descending paths. This relation is established as 
\begin{theorem}\label{thm:mod-limit}
For any $p\in(1,\infty)$, 
\begin{equation*}
\lim_{n\to\infty}\Mod_{p,\sigma}(\Gamma_n)
= \Mod_{p,\sigma}(\Gamma_\infty).
\end{equation*}
\end{theorem}
For radially symmetric infinite trees, this allows us to express the modulus through the series
\begin{equation*}
\Mod_{p,\sigma}(\Gamma_\infty)
= \left(\sum_{k=1}^\infty\left(\sigma_k|S_k|\right)^{-\frac{q}{p}}\right)^{-\frac{p}{q}}
\end{equation*}
where $q$ is the H\"older conjugate exponent of $p$, that is, $\frac{1}{p}+\frac{1}{q}=1$.  This equality can even be interpreted when the series on the right-hand side diverges.  In this case, the modulus is zero.\\

For finite graphs, $p$-modulus has interesting interpretations depending on the choice of the parameter $p$ and the family of objects. The object in the graph is a combinatorial or structural component such as walks, paths, cuts, spanning trees, partitions, stars, and edge covers. For example, as shown in~\cite{AlbinModulusGraphs, Albin2015ModulusQuantities}, the $p$-modulus of a family of paths connecting two vertices, $u$ and $v$, is related to several well-known graph quantities.  The $2$-modulus is the effective conductance between $u$ and $v$, the $1$-modulus is the value of the minimum cut disconnecting $u$ and $v$, and the $\infty$-modulus is the reciprocal of the length of the shortest path between $u$ and $v$.

In this paper, we establish analogous properties for the $p$-modulus of the family of infinite descending paths on a radially symmetric tree. For example, for $p=2$, we establish the connection between \emph{$2$-}modulus and characteristics of random walks. For $p=1$ and $p=\infty$, we derive a special dual problem of $p$-modulus, providing a lower bound. For $p=1$, the modulus in the radially symmetric tree is the infimum over cuts for descending paths. For a cut set $C\subseteq E$, the following relation is established 
\begin{theorem}\label{theorem:min-cut radially}
Let $G=(V,E,\sigma,o)$ be a radially symmetric tree. Then
\begin{equation*}
    \Mod_{1,\sigma}(\Gamma_\infty) =  \inf \{\sigma(C): C \, \text{is a cut of} \,\, \Gamma_\infty\}.
\end{equation*}
\end{theorem}
For $p=\infty$, the modulus is the reciprocal of the weighted length of the family of descending paths. Furthermore, we investigate some properties of the $p$-modulus based on the parameters $p$ and $\sigma$. In the particular case of uniformly elliptic and bounded weights, the modulus in a weighted tree is equivalent to the modulus in an unweighted tree. More importantly, we we establish the relation $\Mod_{1,\sigma}(\Gamma_\infty)>0$ if and only if $\Mod_{p,\sigma}(\Gamma_\infty)>0$ and  
\begin{equation} \label{limit_infty}
\lim_{p\to\infty}\Mod_{p,\sigma}(\Gamma_\infty)^{\frac{1}{p}} = 0 = \Mod_{\infty,\sigma}(\Gamma_\infty).
\end{equation}
Finally, we define the critical exponent for the $p$-modulus on a $1$-$2$ radially symmetric tree as 
\begin{equation}
    p_c=\sup\{p >1 : \Mod_{p,1}(\Gamma_{\infty})>0\}
\end{equation}
Such parameter $p_c$ exists in a proper infinite tree specifically in a $1-2$ radially symmetric binary tree because of $\Mod_{1,1}(\Gamma_{\infty})>0$ or in general, if $\Mod_{p,1}(\Gamma_{\infty})>0$ for some $p$ and by (\ref{limit_infty}). This critical parameter $p_c$ demarcates the boundary between zero and finite modulus. Equivalently, the critical exponent implies the infinite tree is \textit{bushy} or \textit{skinny} in terms of whether a random walk is transient or recurrent. In this sense, the modulus can measure the dimension of an infinite tree. An example in chapter \ref{Chapter:critical_exponent} demonstrates the existence of a $1$-$2$ radially symmetric binary tree of non-trivial critical exponent, and the fact will be established as
\begin{theorem}\label{theorem:existence_critical}
For any $1<r<\infty$, there exists an unweighted 1-2 tree with critical exponent $p_c=r$.
\end{theorem}
In summary, in Section 2, we review notations, definitions, and properties of the modulus, particularly in the context of infinite trees with descending paths. Section 3 presents the computing formula \emph{$p$-}modulus and special dual formulation. In Section 4, we present some properties of \emph{$p$-}modulus in terms of parameters. In Section 5, we examine the existence of a critical parameter that distinguishes between the modulus being positive and zero. Section $7$ ends with a summary and remarks.

\section{$p$-modulus on infinite trees}\label{section-definition of modulus}

Given a tree $G=(V,E, \sigma, o)$, a \emph{density} on $G$ is a function $\rho:E \to \mathbb{R}_{\ge 0}$ on the edge set $E$. It can thought of as a cost function on the edge set. In particular, $\rho(e)$ is the cost per unit usage of edge $e$. For a density $\rho$ and $\gamma \in \Gamma$, we define $\rho$-length of $\gamma$ as 
\begin{equation*}
\ell_{\rho}(\gamma):=\sum_{e \in \gamma} \rho(e), 
\end{equation*}
In particular, $\rho \equiv 1$, the $\rho$-length is the length of path $\gamma$. For a descending family $\Gamma$, the $\rho$-length of $\Gamma$ is defined as   
\begin{equation*}
\ell_{\rho}(\Gamma):= \inf_{\gamma \in \Gamma}\ell_{\rho}(\gamma).
\end{equation*}
A density $\rho$ is \emph{admissible} for $\Gamma$ if
\begin{equation*}
\ell_{\rho}(\gamma) \ge 1 \qquad \forall \gamma \in \Gamma;
\end{equation*}
 and we define the set $\Adm(\Gamma)$ to be set of all admissible densities. That is,  
 \begin{equation*}
 \Adm(\Gamma)=\{\rho \in \mathbb{R}^E_{\ge 0}: \ell_{\rho}(\Gamma) \ge 1\}.
 \end{equation*}
Now, given a real parameter $ 1 \le p \leq \infty$, the \emph{$p$-energy} of a density $\rho$ on $G$ is defined to be
\[
  \mathcal{E}_{p, \sigma}(\rho) =
  \begin{cases}
   \sum\limits_{e\in E}\sigma(e)\rho(e)^p  & \text{if $1 \le p < \infty$} \\
    \max\limits_{e\in E} \sigma(e)\rho(e) & \text{if $p = \infty$}
  \end{cases}
\]
Note that
\begin{equation*}
\lim_{p \to \infty}(\mathcal{E}_{p, \sigma^{p}}(p))^\frac{1}{p}= \max _{e\in E} \sigma(e)\rho(e) = \mathcal{E}_{\infty,\sigma}(\rho)
\end{equation*}
 and both $ \mathcal{E}_{p, \sigma}(\rho)$ and $\ell_{\rho}(\gamma)$ sums over countable or the countably infinite set $E$ and their terms are non-negative.  Thus, in both cases, the partial sums are monotone; each sum either converges to a finite value or diverges.  In the latter case, we assume the value of the sum to be $+\infty$.
\begin{definition}
Given a graph $G=(V,E, \sigma)$ with non-empty family of descending paths $\Gamma$ and exponent $ 1 \le p \leq \infty$, the \emph{$p$-}modulus of $\Gamma$, denoted by $\Mod_p(\Gamma)$ is defined as the value of 
\begin{equation}\label{defn:p-modolus}
\begin{split}
\text{minimize}\qquad& \mathcal{E}_{p,\sigma}(\rho),\\
\text{subject to}\qquad&\ell_{\rho}(\gamma) \ge 1 \qquad  \forall\, \gamma \in \Gamma; \\
\qquad& \rho(e)\ge 0 \qquad  \forall \, e \in E.
\end{split}
\end{equation}
\end{definition}
A density $\rho^* \in \Adm(\Gamma)$ is called \emph{extremal} or optimal density if (\ref{defn:p-modolus}) is minimized at $\rho^*$ or  $\mathcal{E}_{p,\sigma}(\rho^*)=\Mod_{p,\sigma}(\Gamma)$. If the graph $G$ is finite, then the $p$-modulus ~(\ref{defn:p-modolus}) on the graph is a standard convex optimization problem, and the unique extremal exists for $1<p < \infty$ (see \cite{AlbinModulusGraphs}). For the infinite setting in~(\ref{defn:p-modolus}), we must take a little care when dealing with the infinite sums. 

Let $G=(E,V, \sigma,o)$ be a proper infinite tree and $\Gamma$, $\Gamma_n$ and $\Gamma_\infty$ be the families of descending paths described in the Section~\ref{Introduction}. Let $n$ be any generation and define
\begin{equation*}
\rho(e) = 
\begin{cases}
 1 & \text{if }e\in S_n,\\
 0 & \text{if }e\notin S_n.
\end{cases}
\end{equation*}
This density is admissible since every $\gamma\in\Gamma_\infty$ contains an edge in $S_n$; hence, its energy provides an upper bound on the modulus. Taking the infimum of energy (~\ref{defn:p-modolus}) over all generations establishes the lemma, providing an upper modulus bound. 
\begin{lemma}\label{lemma:p2bounded}
For all $p\in[1,\infty]$,
\begin{equation*}
\Mod_{p,\sigma}(\Gamma_\infty)\le \inf_n\sigma(n)|S_n|.
\end{equation*}
\end{lemma}
Our goal to prove that $\Mod_{p, \sigma}(\Gamma_{\infty})$ is a natural extension of $\Mod_{p, \sigma}(\Gamma_{n})$ as $n$ become arbitrary large. For this, first, we prove the following lemmas.
\begin{lemma}\label{lemma:infinitelength}
Let $\rho\in\mathbb{R}^E_{\ge 0}$.  Then
\begin{equation*}
\lim_{n\to\infty}\ell_\rho(\Gamma_n) = \ell_\rho(\Gamma_\infty).
\end{equation*}
\end{lemma}
\begin{proof}
Note that $\{\ell_\rho(\Gamma_n)\}$ is a monotone sequence, so either the limit exists as a finite number or the sequence diverges to $+\infty$, in which case we say that the limit is $+\infty$.
Let $A_0=\{\gamma_n\}$ be a sequence of paths with $\gamma_n\in\Gamma_n$ such that $\ell_\rho(\gamma_n)=\ell_\rho(\Gamma_n)$.  By the assumption of local finiteness, there must be an edge $e_1\in S_1$ that intersects infinitely many paths in $A_0$.  Let $A_1\subseteq A_0$ be the infinite subsequence produced by removing all paths that do not pass through $e_1$.  Again, some edge $e_2\in S_2$ must intersect infinitely many paths in $A_1$.  Let $A_2\subseteq A_1$ be the infinite subsequence produced by removing all paths that do not use this edge.  Repeating this procedure, we may produce an infinite sequence of edges $\{e_1,e_2,e_3,\ldots\}$ and a chain of infinite subsets $A_1\supseteq A_2\supseteq A_3\supseteq\cdots$ with the property that every path in $A_n$ begins with the subpath $e_1e_2\cdots e_n$ and every path in $A_n$ is a shortest-length path in some $\Gamma_m$ with $m\ge n$.
Now consider the infinite path $\gamma=e_1e_2e_3\cdots$.  For any $n\ge 1$, there exists an $m\ge n$ and a $\gamma'\in A_n\cap \Gamma_m$ with the property that $\ell_\rho(\gamma')=\ell_\rho(\Gamma_m)$.  So,
\begin{equation*}
\ell_\rho(\gamma\cap B_n) = \ell_\rho(\gamma'\cap B_n) \le \ell_\rho(\gamma') = \ell_\rho(\Gamma_m).
\end{equation*}
Since $\ell_\rho(\gamma\cap B_n)\to \ell_\rho(\gamma)$ as $n\to\infty$, this shows that
\begin{equation*}
\lim_{n\to\infty}\ell_\rho(\Gamma_n) \ge \ell_\rho(\gamma) \ge \ell_\rho(\Gamma_\infty).
\end{equation*}

On the other hand, if $\gamma\in\Gamma_\infty$ then for any $n\ge 1$,
\begin{equation*}
\ell_\rho(\Gamma_n) \le \ell_\rho(\gamma\cap B_n) \le \ell_\rho(\gamma).
\end{equation*}
Taking the infimum over $\gamma\in\Gamma_\infty$ establishes the opposite inequality.
\end{proof}
\begin{lemma}\label{lem:admissible-nesting}
If $1\le m\le n\le\infty$, then $\Adm(\Gamma_m)\subseteq\Adm(\Gamma_n)$.
\end{lemma}
\begin{proof}
Let $\rho\in\Adm(\Gamma_m)$ and let $\gamma\in\Gamma_n$.  By the sub-path relationship, there is sub-path $\gamma' \in \Gamma_m$ such that $\gamma' \preceq \gamma$, that is, $\gamma$ is an extension of $\gamma'$.  Thus, since $\rho\ge 0$,
\begin{equation*}
\ell_{\rho}(\gamma) = \sum_{e\in\gamma}\rho(e) \ge \sum_{e\in\gamma'}\rho(e) = \ell_{\rho}(\gamma')
\ge 1.
\end{equation*}
Thus, $\rho\in\Adm(\Gamma_n)$.
\end{proof}

\begin{proof}[Proof of Theorem~\ref{thm:mod-limit}]
By the Lemma~\ref{lem:admissible-nesting}, the sequence $\{\Mod_{p,\sigma}(\Gamma_n)\}_{n=1}^\infty$ is a non-increasing sequence of non-negative numbers and, therefore, approaches a limit.  Now let $\epsilon>0$ be arbitrary and let $\rho\in\Adm(\Gamma_\infty)$ be an improper infinite tree and have the property that
\begin{equation*}
\mathcal{E}_{p, \sigma}(\rho) \le \Mod_{p, \sigma}(\Gamma_\infty)+\epsilon.
\end{equation*}
Since $\ell_\rho(\Gamma_\infty)\ge 1$ by assumption, Lemma~\ref{lemma:infinitelength} implies that $\ell_\rho(\Gamma_n)>0$ for sufficiently large $n$.  For these $n$, it is straightforward to check that $\rho/\ell_\rho(\Gamma_n)\in\Adm(\Gamma_n)$.  

It follows that for sufficiently large $n$,
\begin{equation*}
\Mod_{p, \sigma}(\Gamma_n) \le \mathcal{E}_{p, \sigma}\left(\frac{\rho}{\ell_\rho(\Gamma_n)}\right)=\sum_{e\in E} \sigma(e)\left(\frac{\rho(e)}{\ell_\rho(\Gamma_n)}\right )^p=\frac{\mathcal{E}_{p, \sigma}(\rho)}{\ell_\rho(\Gamma_{n})^p}.
\end{equation*}
Taking a limit as $n\to\infty$ shows that $\lim\limits_{n \to \infty}\Mod_{p,\sigma}(\Gamma_n) \le \Mod_{p,\sigma}(\Gamma_\infty) + \epsilon$.  Since $\epsilon>0$ was arbitrary, we have 
\begin{equation}\label{eq:lessthaninfinitemodulus}
    \lim\limits_{n \to \infty}\Mod_{p,\sigma}(\Gamma_n)\le\Mod_{p,\sigma}(\Gamma_\infty)
\end{equation}
On the other hand, Lemma~\ref{lem:admissible-nesting} implies the opposite inequality
\begin{equation} \label{eq:greaterthaninfinitemodulus}
    \Mod_{p,\sigma}(\Gamma_\infty) \le \lim\limits_{n \to \infty}\Mod_{p,\sigma}(\Gamma_n).
\end{equation}
Inequalities (\ref{eq:lessthaninfinitemodulus}) and (\ref{eq:greaterthaninfinitemodulus}) combined prove the theorem.
\end{proof}

\section{Computing modulus on radially symmetric tree}
Given a radially symmetric tree, it is natural to consider densities $\rho$ that share the tree's symmetry.  We shall call $\rho$ \emph{radially symmetric} if $\rho(e)=\rho(\gen(e))$---that is, if $\rho(e)$ depends only on the generation of $e$. The following lemma guarantees that it is sufficient to consider radially symmetric densities for the modulus on a radially symmetric tree. 
\begin{lemma}\label{lem:rad-sym}
Suppose $\rho\in\Adm(\Gamma_n)$ for some $1\le n\le\infty$ and that $\rho$ is not radially symmetric.  Then there exists a $\rho'\in\Adm(\Gamma_n)$ such that
\begin{equation*}
    \mathcal{E}_{p,\sigma}(\rho') < \mathcal{E}_{p,\sigma}(\rho).
\end{equation*}
\end{lemma}
\begin{proof}
Let $m\le n$ be the earliest generation where $\rho$ fails to be radially symmetric and let $e,e'\in S_m$ be such that $\rho(e)\ne\rho(e')$.  Define $\tilde{\rho}$ from $\rho$ by swapping the values on $e$ and its children with the corresponding values on $e'$ and its children.  This density is still admissible and has the same energy as $\rho$.  Now define $\rho'=(\rho+\tilde{\rho})/2$.  As an average of two admissible densities, $\rho'$ is also admissible.  Moreover, by the strict convexity of the energy, it follows that
\begin{equation*}
\mathcal{E}_{p,\sigma}(\rho') <
\frac{1}{2}\mathcal{E}_{p,\sigma}(\tilde{\rho}) +
\frac{1}{2}\mathcal{E}_{p,\sigma}(\rho)
= \mathcal{E}_{p,\sigma}(\rho).
\end{equation*}
\end{proof}
In the case of a radially symmetric tree, the Lemma~\ref{lemma:infinitelength} is more easily established when the density $\rho$ is radially symmetric. For computing the $p$-modulus in a radially symmetric tree, we treat the cases $1<p<\infty$, $p=1$ and $p=\infty$ separately.  Specially, we use several facts about an exponent $p\in(1,\infty)$ and its corresponding H\"older conjugate exponent $q$:
\begin{equation*}
\frac{1}{p}+\frac{1}{q}=1
\end{equation*}

For $1<p<\infty$, the modulus (\ref{defn:p-modolus}) of a truncated family $\Gamma_n$ with $1\le n < \infty$ could be expressed as
\begin{equation}\label{def:p-modulus-finite}
\begin{split}
\text{minimize}\qquad&\sum_{k=1}^n|S_k|\sigma_k\rho_{n,k}^p,\\
\text{subject to}\qquad&\sum_{k=1}^n\rho_{n,k}\ge 1.
\end{split}
\end{equation}
where $\sigma_k$ and $\rho_{n,k}$ the values of $\sigma$ and $\rho_n$ on $S_k$ respectively on a radially symmetric tree. In this case, the problem (\ref{def:p-modulus-finite}) is a standard convex optimization problem and has a unique optimal density (see, \cite{AlbinModulusGraphs, Albin2015ModulusQuantities}). By using standard optimization techniques, we solve for the optimal density for ($\ref{def:p-modulus-finite}$), and the minimum value is 
\begin{equation*}
\Mod_{p,\sigma}(\Gamma_n) =
\left(\sum_{k=1}^n\left(\sigma_k|S_k|\right)^{-\frac{q}{p}}\right)^{-\frac{p}{q}}
\end{equation*}
where optimal density $\rho_{n,k}$ is given by
\begin{equation*}
\rho_{n,k} = 
\frac{\left(\sigma_k|S_k|\right)^{-\frac{q}{p}}}
{\sum\limits_{\ell=1}^n\left(\sigma_\ell|S_\ell|\right)^{-\frac{q}{p}}}.
\end{equation*}
By the Theorem~\ref{thm:mod-limit} and taking the limit, we find a formula for the modulus of the infinite tree:
\begin{equation}\label{eq:mod-formula-inf-family}
\Mod_{p,\sigma}(\Gamma_\infty)
= \left(\sum_{k=1}^\infty\left(\sigma_k|S_k|\right)^{-\frac{q}{p}}\right)^{-\frac{p}{q}},
\end{equation}
The formula (\ref{eq:mod-formula-inf-family}) makes modulus much easier for calculation in terms of convergence and divergence of series. If the quantities $\sigma_k|S_k|$ grow sufficiently fast, then the series converges and hence $\Mod_{p,\sigma}(\Gamma_\infty)>0$. In this case, the optimal density exists. Otherwise, $\Mod_{p,\sigma}(\Gamma_\infty)=0$ and there is no optimal density. Therefore, the existence or non-existence of an optimal density is closely related to the convergence or divergence of the infinite sum in (\ref{eq:mod-formula-inf-family}). In particular, $p$-modulus can be expressed as an infinite sum
\begin{equation}\label{eq:2-mod-formula}
    \Mod_{2,\sigma}(\Gamma_\infty)
= \left(\sum_{k=1}^\infty\frac{1}{\sigma_k|S_k|}\right)^{-1}
\end{equation}

\begin{theorem}\label{thm:mod-opt-den} 
For $1 <p < \infty$, an optimal density for $\Mod_{p,\sigma}(\Gamma_\infty)$ exists if and only if $\Mod_{p,\sigma}(\Gamma_\infty)>0$.
\end{theorem}
\begin{proof}
If the modulus is zero, then no optimal density can exist.  Suppose, on the other hand, that the modulus is positive. Then the series
\begin{equation*}
\sum_{k=1}^\infty\left(\sigma_k|S_k|\right)^{-\frac{q}{p}}
\end{equation*}
converges.

Define a radially symmetric density $\rho$ so that on $S_k$, it takes the value
\begin{equation}\label{eq:opt-symm-rho}
\rho_k :=
\frac{\left(\sigma_k|S_k|\right)^{-\frac{q}{p}}}
{\sum\limits_{\ell=1}^\infty\left(\sigma_\ell|S_\ell|\right)^{-\frac{q}{p}}}.
\end{equation}
Note that $\rho$ is admissible, since $\sum\limits_{k=1}^\infty\rho_k=1$.  Moreover,
\begin{equation*}\label{eq:symm-mod-formula}
\begin{split}
\mathcal{E}_{p,\sigma}(\rho) &=
\sum_{k=1}^\infty
\frac{\sigma_k|S_k|\left(\sigma_k|S_k|\right)^{-q}}
{\left(\sum\limits_{\ell=1}^\infty\left(\sigma_\ell|S_\ell|\right)^{-\frac{q}{p}}\right)^p}
=
\frac{\sum\limits_{k=1}^\infty\left(\sigma_k|S_k|\right)^{-\frac{q}{p}}}
{\left(\sum\limits_{\ell=1}^\infty\left(\sigma_\ell|S_\ell|\right)^{-\frac{q}{p}}\right)^p}\\
&= \left(\sum_{k=1}^\infty\left(\sigma_k|S_k|\right)^{-\frac{q}{p}}\right)^{-\frac{p}{q}} \\
& = \Mod_{p,\sigma}(\Gamma_\infty)
\end{split}
\end{equation*}
showing that $\rho$ is optimal.
\end{proof}
By the Theorem \ref{thm:mod-limit}, the effective conductance on the proper infinite tree is $\Mod_{2,\sigma}(\Gamma_\infty)$. Hence, the Theorem~\ref{thm:mod-opt-den} has immediate consequences that relate the $2$-modulus with the nature of a simple random walk in an infinite tree (see,\cite{LyonsProbabilityNetworks, Vatamanelu2010TheNetworkb}). 
   A random walk is transient or recurrent depending on how the number of children in each generation grows. Specially, a random walk in a radially symmetric tree whose edges each have more than one child is always transient.
 \begin{corollary}\label{corollary:transient}
 A random walk in an infinite tree $G$ is transient if and only if $\Mod_{2,1}(\Gamma_\infty)>0$.
 \end{corollary}
  \begin{corollary}
 The random walk on an unweighted infinite tree is transient if and only if $\sum\limits_{k=1}^\infty\frac{1}{|S_k|} < \infty$.
 \end{corollary}
By~\eqref{defn:p-modolus}, for $p=1$,  the \emph{$1$-modulus} of a family of descending paths in a radially symmetric tree is defined as 
\begin{equation} \label{def:1-modulus}
\begin{aligned}
\minimize  \quad & \sum_{e\in E}\sigma(e)\rho(e)\\ 
\sub \quad & \ell_{\rho}(\gamma) \ge 1 \qquad  \forall\, \gamma \in \Gamma_\infty, \\
& \rho(e)\ge 0 \qquad  \forall \, e \in E
\end{aligned}
\end{equation}
and for $p=\infty$, the \emph{$\infty$-modulus} of a family of descending paths in a proper infinite tree defined as 
\begin{equation*}
\begin{aligned}
\minimize \quad  & \sup_{e \in E} \sigma(e)\rho(e)\\ 
\sub \quad & \ell_{\rho}(\gamma) \ge 1 \qquad  \forall\, \gamma \in \Gamma_\infty, \\
& \rho(e)\ge 0 \qquad  \forall \, e \in E.
\end{aligned}
\end{equation*}
In both cases of $p$, $\rho$-energy is not strictly convex, meaning that an optimal density (if it exists) need not be unique. Before establishing a result of $1$-modulus and $\infty$-modulus related to other physical quantities analogous to finite setting (see, \cite{Albin2015ModulusQuantities}), we derive a special dual formulation for $p$-modulus, and it helps to establish lower bounds on modulus by establishing a dual problem.  Applying classical Lagrangian duality to the $p$-modulus problem is difficult in part because there are uncountably many inequality constraints. In a sense, the following construction replaces the classical dual by providing a family of lower bounds for $p$-modulus.  The largest such lower bound can be characterized as a type of dual problem, and, at least in the case of radially symmetric trees, the lower bound is exact.

We begin by defining
\begin{equation*}
\Lambda := \left\{\eta\in\mathbb{R}^E_{\ge 0}:\sum_{e\in S_1}\eta(e)=1,\;\sum_{e'\in c(e)}\eta(e')=\eta(e) \quad \forall e\in E\right\}.
\end{equation*}\label{eq:lambda}
We observe that it is not hard to show the set $\Lambda$ is a convex subset of $\mathbb{R}_{\ge 0}^{E}$. The following lemma is an immediate consequence of the set $\Lambda$ that if $\lambda \in \Lambda$, then it assigns unit mass to all shells in the tree.

\begin{lemma} \label{lem:eta-sum-1}
Let $\eta\in\Lambda$ and let $n\ge 1$.Then
\begin{equation*}
\eta(S_n) = \sum_{e\in S_n}\eta(e) = 1.
\end{equation*}
\end{lemma}
\begin{lemma}\label{lem:rho-eta-Bn}
Let $\rho\in\mathbb{R}^E_{\ge 0}$ and let $\eta\in\Lambda$.  Then, for any $n\ge 1$
\begin{equation*}
\sum_{e\in B_n}\rho(e)\eta(e) = \sum_{e\in S_n}\ell_\rho(\gamma_e)\eta(e),
\end{equation*}
where $\gamma_e$ is the path descending from the root and terminating after passing through the edge $e$.
\end{lemma}
\begin{proof}
The equality is trivially true when $n=1$.  Suppose it is true for some $n\ge 1$, then
\begin{equation*}
\begin{split}
\sum_{e\in B_{n+1}}\rho(e)\eta(e) &= \sum_{e\in B_n}\rho(e)\eta(e) + \sum_{e\in S_{n+1}}\rho(e)\eta(e)\\
&= \sum_{e\in S_n}\ell_\rho(\gamma_e)\eta(e) + \sum_{e\in S_n}\sum_{e'\in c(e)}\rho(e')\eta(e')\\
&= \sum_{e\in S_n}\ell_\rho(\gamma_e)\sum_{e'\in c(e)}\eta(e') + \sum_{e\in S_n}\sum_{e'\in c(e)}\rho(e')\eta(e')\\
&= \sum_{e\in S_n}\sum_{e'\in c(e)}\left(\ell_\rho(\gamma_e)+\rho(e')\right)\eta(e')\\
&= \sum_{e\in S_n}\sum_{e'\in c(e)}\ell_\rho(\gamma_{e'})\eta(e')\\
&= \sum_{e\in S_{n+1}}\ell_\rho(\gamma_{e})\eta(e).
\end{split}
\end{equation*}
\end{proof}
The following theorem establishes a connection between the modulus and the set $\Lambda$ by characterizing the admissible set $\Adm(\Lambda)$.
\begin{theorem}\label{thm:rho-eta-inequality}
Let $\rho\in\mathbb{R}^E_{\ge 0}$.  Then $\rho\in\Adm(\Gamma_\infty)$ if and only if
\begin{equation}\label{eq:rho-eta-inequality}
\sum_{e\in E}\rho(e)\eta(e)\ge 1\qquad\text{for all }\eta\in\Lambda.
\end{equation}
\end{theorem}
\begin{proof}
Suppose $\rho\in\Adm(\Gamma_\infty)$.  Then for any $\eta\in\Lambda$, Lemma~\ref{lem:rho-eta-Bn} implies that
\begin{equation*}
\sum_{e\in E}\rho(e)\eta(e) = \lim_{n\to\infty}\sum_{e\in B_n}\rho(e)\eta(e)
=\lim_{n\to\infty}\sum_{e\in S_n}\ell_\rho(\gamma_e)\eta(e).
\end{equation*}
Since $\gamma_e\in\Gamma_n$ for $e\in S_n$, it follows that $\ell_\rho(\gamma_e)\ge\ell_\rho(\Gamma_n)$.  Thus, by Lemma~\ref{lem:eta-sum-1}, the sum on the right is a convex combination of values, each bounded below by $\ell_\rho(\Gamma_n)$.  Thus, by Lemma~\ref{lemma:infinitelength},
\begin{equation*}
\sum_{e\in E}\rho(e)\eta(e) 
\ge \lim_{n\to\infty}\ell_\rho(\Gamma_n) = \ell_\rho(\Gamma_\infty)\ge 1.
\end{equation*}

On the other hand, suppose $\rho\notin\Adm(\Gamma_\infty)$.  Then, there exists $\gamma\in\Gamma_\infty$ such that $\ell_\rho(\gamma)<1$.  Define $\eta$ to be the indicator function of $\gamma$: $\eta = \mathbf{1}_\gamma$.  Note that $\eta\in\Lambda$.  Moreover,
\begin{equation*}
\sum_{e\in E}\rho(e)\eta(e) = \sum_{e\in\gamma}\rho(e) = \ell_\rho(\gamma) < 1.
\end{equation*}
\end{proof}
The following theorem characterizes the set $\Lambda$ and provides the lower bounds for $p$-modulus.
\begin{theorem}\label{thm:mod-lower-bound-eta}
Let $\eta\in\Lambda$, then
\begin{equation*}
\Mod_{p,\sigma}(\Gamma_\infty) \ge 
\begin{cases}
\left(\sum\limits_{e\in E}\sigma(e)^{-\frac{q}{p}}\eta(e)^q\right)^{-\frac{p}{q}} & \text{if }p\in(1,\infty),\\
\left(\sup\limits_{e\in E}\sigma(e)^{-1}\eta(e)\right)^{-1} & \text{if }p=1,\\
\left(\sum\limits_{e\in E}\sigma(e)^{-1}\eta(e)\right)^{-1} & \text{if }p=\infty.
\end{cases}
\end{equation*}
\end{theorem}
\begin{proof}
Let $\rho\in\Adm(\Gamma_\infty)$.  First, consider the case $p\in(1,\infty)$.  By Theorem~\ref{thm:rho-eta-inequality} and H\"older's inequality, we have that
\begin{equation*}
\begin{split}
1 &\le \sum_{e\in E}\rho(e)\eta(e) = \sum_{e\in E}\sigma(e)^{\frac{1}{p}}\rho(e)\sigma(e)^{-\frac{1}{p}}\eta(e) \\
&\le \left(\sum_{e\in E}\sigma(e)\rho(e)^p\right)^\frac{1}{p}
\left(\sum_{e\in E}\sigma(e)^{-\frac{q}{p}}\eta(e)^q\right)^\frac{1}{q},
\end{split}
\end{equation*}
from which it follows that
\begin{equation*}
\mathcal{E}_{p,\sigma}(\rho)\ge\left(\sum_{e\in E}\sigma(e)^{-\frac{q}{p}}\eta(e)^q\right)^{-\frac{p}{q}}.
\end{equation*}
For $p=1$,
\begin{equation*}
    1 \le \sum_{e\in E}\rho(e)\eta(e)=\sum_{e\in E}\sigma(e)\rho(e)\left(\sigma(e)^{-1}\eta(e)\right) \le \left(\sum_{e\in E}\sigma(e)\rho(e)\right)\left(\sup_{e \in E}\sigma(e)^{-1}\eta(e)\right)
\end{equation*}
And for $p=\infty$, 
\begin{equation*}
    1 \le \sum_{e\in E}\sigma(e)\rho(e)\left(\sigma(e)^{-1}\eta(e)\right) \le \left(\sup_{e \in E}\sigma(e)\rho(e)\right)\left(\sum_{e \in E}\sigma(e)^{-1}\eta(e)\right)
\end{equation*}
Taking the infimum over all admissible densities yields the desired inequality.
\end{proof}
From the Theorem~\ref{thm:mod-lower-bound-eta}, we have immediate two important consequences.
\begin{corollary}
Suppose $\Mod_{p,\sigma}(\Gamma_\infty)>0$ for $ 1<p< \infty$, then  \begin{equation*}
    \Mod_{p,\sigma}(\Gamma_\infty)\ge \left(\inf\limits_{\eta \in \Lambda}\sum_{e\in E}\sigma(e)^{-\frac{q}{p}}\eta(e)^q\right)^{-\frac{p}{q}}
\end{equation*}
Moreover, equality is attained for a radially symmetric infinite tree, and the optimal density is 
\begin{equation*}
    \rho_k^*=\frac{\left(\sigma_k\eta_k^{-1}\right)^{-\frac{q}{p}}}{\sum_{\ell=1}^\infty\left(\sigma_k\eta_k^{-1}\right)^{-\frac{q}{p}}}
\end{equation*}
\end{corollary}
\begin{corollary}\label{coro:lowerboundMod1}
Let $G=(V,E,\sigma,o)$ be a radially symmetric infinite tree.  Then
\begin{equation*}
\Mod_{1,\sigma}(\Gamma_\infty) \ge \inf_{k\ge 1}\sigma(S_k).
\end{equation*}
\end{corollary}
Finally, we arrived to establish the results of the Theorem~\ref{theorem:min-cut radially}. The $1$-modulus is the infimum of cuts for the infinite tree, and the $\infty$-modulus is the reciprocal of the weighted length of the family of descending paths. Let $C$ be a cut set for $\Gamma_\infty$ in a radially symmetric binary tree. 

\begin{proof}[Proof of Theorem~\ref{theorem:min-cut radially}]
Let $C \subseteq E$ be a cut set for $\Gamma_\infty$. Define a density $\rho$ on $E$ as follows
\begin{equation*}
 \rho(e)= \begin{cases}
 1 & \text{if } e \in C,\\
 0 & \text{otherwise}.
            \end{cases} 
\end{equation*} 
It is straightforward to show that $\rho \in \Adm(\Gamma_\infty)$ since every descending path crosses at least one edge of $C$. Moreover, 
\begin{equation*}
    \Mod_{1,\sigma}(\Gamma_\infty)\le 
    \mathcal{E}_{1,\sigma}(\rho) =
    \sum_{e \in E}\sigma(e)\rho(e)=\sum_{e \in C}\sigma(e) =\sigma(C).
\end{equation*}
Taking the infimum shows that
\begin{equation*}
    \Mod_{1,\sigma}(\Gamma_\infty) \le  \inf \{\sigma(C): C \, \text{is a cut of} \,\, \Gamma_\infty\}.
\end{equation*}
On the other hand, the shell $S_k$ is a cut for the family of descending paths for all $k$. So, as a consequence of Corollary~\eqref{coro:lowerboundMod1}, we have
\begin{equation*}
\Mod_{1, \sigma}(\Gamma_\infty) \ge
\inf\limits_{k \ge 1}\sigma(S_k)
\ge \inf\{\sigma(C):C \, \text{is a cut for }\Gamma_\infty\}.
\end{equation*}
\end{proof}
\begin{theorem}\label{thm:mod_infity}
Let $G=(V,E,\sigma,o)$ be a radially symmetric tree. Then 
\begin{equation*}
    \Mod_{\infty,\sigma}(\Gamma_\infty) =
\begin{cases}
\frac{1}{\ell_{\sigma^{-1}}(\Gamma_\infty)} &\text{if }\ell_{\sigma^{-1}}(\Gamma_\infty)<\infty,\\
0 &\text{otherwise}.
\end{cases}
\end{equation*}
\end{theorem}
\begin{proof}
Let $\eta\in\Lambda$.  From Lemma~\ref{lem:rho-eta-Bn} and~\eqref{lemma:infinitelength}, it follows that
\begin{equation*}
\sum_{e\in E}\sigma(e)^{-1}\eta(e) =
\lim_{n\to\infty}\sum_{e\in B_n}\sigma(e)^{-1}\eta(e) =
\lim_{n\to\infty}\sum_{e\in S_n}\ell_{\sigma^{-1}}(\gamma_e)\eta(e)
=\lim_{n\to\infty}\ell_{\sigma^{-1}}(\Gamma_n)
= \ell_{\sigma^{-1}}(\Gamma_\infty).
\end{equation*}

First, suppose that $\ell_{\sigma^{-1}}(\Gamma_\infty)<0$.  Then, Theorem~\ref{thm:mod-lower-bound-eta} shows
\begin{equation}\label{eq:infinity_mode}
    \Mod_{\infty, \sigma}(\Gamma_\infty) \ge \left( \sum_{e\in E}\sigma(e)^{-1}\eta(e)\right)^{-1} \ge \left(\ell_{\sigma^{-1}}(\Gamma_\infty)\right)^{-1}
\end{equation}
Define the density 
\begin{equation*}
    \rho(e)=
    \sigma(e)^{-1}(\ell_{\sigma^{-1}}(\Gamma_\infty))^{-1}
\end{equation*}
 for all $e\in E$. The density $\rho$ is admissible since for $\gamma \in \Gamma_\infty$,
\begin{equation*}
    \ell_{\rho}(\gamma)=\sum_{e\in \gamma}\rho(e)=\frac{1}{\ell_{\sigma^{-1}}(\Gamma_\infty)}\sum_{e\in \gamma}\sigma(e)^{-1}= 1.
\end{equation*}
Now, 
\begin{equation*}
    \sigma(e)\rho(e)=\frac{1}{\ell_{\sigma^{-1}}(\Gamma_\infty)}  .
\end{equation*}
and 
\begin{equation*}
    \sup_{e \in E}\sigma(e)\rho(e) = \frac{1}{\ell_{\sigma^{-1}}(\Gamma_\infty)}  
\end{equation*}
This shows that
\begin{equation}\label{eq:mod_infty1}
    \Mod_{\infty, \sigma}(\Gamma_\infty) \le \sup_{e \in E}\sigma(e)\rho(e) = \frac{1}{\ell_{\sigma^{-1}}(\Gamma_\infty)}.
\end{equation}
Hence, by~\eqref{eq:infinity_mode} and ~\eqref{eq:mod_infty1}
\begin{equation*}
    \Mod_{\infty, \sigma}(\Gamma_\infty)= \frac{1}{\ell_{\sigma^{-1}}(\Gamma_\infty)}.
\end{equation*}

On the other hand, suppose $\ell_{\sigma^{-1}}(\Gamma_\infty)=\infty$.  Let $\epsilon>0$ and consider the density $\rho(e)=\epsilon\sigma(e)^{-1}$.  Since
\begin{equation*}
\ell_\rho(\Gamma_\infty)=\epsilon\ell_{\sigma^{-1}}(\Gamma_\infty) = \infty,
\end{equation*}
this density is admissible.  Thus,
\begin{equation*}
\Mod_{\infty,\sigma}(\Gamma_\infty) \le \mathcal{E}_{\infty,\sigma}(\rho) = \epsilon.
\end{equation*}
Since $\epsilon>0$ is arbitrary, the modulus is zero.
\end{proof}
If $\sigma \equiv 1$ and $\ell_{\sigma^{-1}}(\Gamma_{\infty})$ becomes arbitrary large, then the Theorem~\ref{thm:mod_infity} an immediate result.
\begin{corollary}\label{cor:infity_modulus}

In the limit,
\begin{equation*}\label{eq:infinitymodulus_as_limit}
\lim_{p\to\infty}\Mod_{p,1}(\Gamma_\infty)^{\frac{1}{p}} = 0=\Mod_{\infty,1}(\Gamma_\infty).
\end{equation*}
\end{corollary}
In the case of an unweighted proper tree, each shell $S_k$ is a cut for the family of descending paths. The following lemma shows that the first shell $S_1$ is a minimum cut for the family of descending paths.
\begin{lemma}\label{lem:cut-inequality}
The $1$-modulus satisfies
\begin{equation*}
\Mod_{1,1}(\Gamma_\infty) = |S_1|.
\end{equation*}
\end{lemma}
\begin{proof}
First, note that the density $\rho_1$, which assigns $1$ to every edge in $S_1$ and $0$ to every other edge, is admissible.  So
\begin{equation*}
\Mod_{1,1}(\Gamma_\infty) \le \mathcal{E}_{1,1}(\rho_1) = |S_1|.
\end{equation*}

On the other hand, suppose $\rho$ is admissible.  Let $\{e_1,e_2,e_3,\ldots,e_r\}$ be an enumeration of $S_1$ with $r=|S_1|$, and let $\gamma_i$ be an arbitrary path in $\Gamma_\infty$ that passes through $e_i$ for $i=1,2,3,\ldots,r$.  Note that the paths $\{\gamma_i\}$ are pairwise disjoint.  Thus,
\begin{equation*}
\mathcal{E}_{1,1}(\rho) = \sum_{e\in E}\rho(e) \ge \sum_{e\in\bigcup\limits_{i=1}^r\gamma_i}\rho(e)
= \sum_{i=1}^r\sum_{e\in\gamma_i}\rho(e) \ge r=|S_1|.
\end{equation*}
Since this is true from an arbitrary admissible $\rho$, it follows that $\Mod_{1,1}(\Gamma_\infty)\ge |S_1|$.  Since the opposite inequality was already established, equality holds.
\end{proof}

\section{Some properties}
In \eqref{eq:mod-formula-inf-family} the $\Mod_{p, \sigma}(\Gamma_\infty)$  is expressed as a reciprocal of an infinite series. Therefore, $\Mod_{p, \sigma}(\Gamma_\infty)$ is either zero or some positive finite value depending on whether the series converges or diverges. For given weight $\sigma$, if the $\Mod_{\sigma, p}(\Gamma_{\infty}) >0$, then the $p$-modulus holds monotonicity property in parameter $p$. Let $p'$ and $q'$ be the conjugate of $p$ and $q$, respectively with $1<p\le p'< \infty$. The series~\eqref{eq:mod-formula-inf-family} holds the inequality
\begin{equation*}
\sum_{k=1}^\infty\left(\sigma_k|S_k|\right)^{-\frac{q}{p}} \le \sum_{k=1}^\infty\left(\sigma_k|S_k|\right)^{-\frac{q'}{p'}} .
\end{equation*}
Raising the power to negative of $p/q$ both side of inequality yields the following lemma
\begin{lemma}\label{lemma:symmetric_pmono}
Let $1<p\le p'< \infty$ and $\Mod_{p,\sigma}(\Gamma_\infty)>0$. Then
\begin{equation*}
    \Mod_{p',\sigma}\left(\Gamma_\infty\right)\le \Mod_{p,\sigma}\left(\Gamma_\infty\right).
\end{equation*}
\end{lemma}
The $p$-modulus and optimal density as a function of parameter $p$ is continuous.

\begin{theorem}\label{thm:symm-mod-cont}
Suppose $\Mod_{p,\sigma}(\Gamma_\infty)>0$ for all $p$ in a neighborhood of $p_0\in(1,\infty)$, then the function $p\mapsto\Mod_{p,\sigma}(\Gamma_\infty)$ is continuous at $p_0$.
\end{theorem}
\begin{proof}
Suppose the modulus is positive in a closed interval $[p_0-\epsilon,p_0+\epsilon]\subset(1,\infty)$ with $\epsilon>0$ and let $p$ belong to this interval.  The fact that $\Mod_{p,\sigma}(\Gamma_\infty)>0$ implies that the infinite sum converges.  Let $q_0$ and $q$ be the H\"older conjugate exponents of $p_0$ and $p$, respectively.  The limit test implies that $\sigma_k|S_k|$ diverge to $+\infty$ as $k\to\infty$.  Thus, there exists a constant $c>0$ such that $\sigma_k|S_k|\ge c$ for all $k$.  From~\eqref{eq:mod-formula-inf-family}, we have
\begin{equation}\label{eq:mod-p-pre-tannery}
\Mod_{p,\sigma}(\Gamma_\infty)^{-\frac{1}{p-1}}
= \sum_{k=1}^\infty\left(\sigma_k|S_k|\right)^{-\frac{1}{p-1}}
= c^{-\frac{1}{p-1}}\sum_{k=1}^\infty\left(\frac{\sigma_k|S_k|}{c}\right)^{-\frac{1}{p-1}}.
\end{equation}
To pass to the limit as $p\to p_0$, we verify the conditions of Tannery's theorem \cite{loyapaul}.

Since $\sigma_k|S_k|/c\ge 1$ for all $k$, we have that
\begin{equation*}
\left(\frac{\sigma_k|S_k|}{c}\right)^{-\frac{1}{p-1}} \le
\left(\frac{\sigma_k|S_k|}{c}\right)^{-\frac{1}{p_0+\epsilon-1}}.
\end{equation*}
By assumption,
\begin{equation*}
\sum_{k=1}^\infty\left(\frac{\sigma_k|S_k|}{c}\right)^{-\frac{1}{p_0+\epsilon-1}}
= c^{\frac{1}{p_0+\epsilon-1}}\Mod_{p_0+\epsilon,\sigma}(\Gamma_\infty)^{-\frac{1}{p_0+\epsilon-1}}.
\end{equation*}
Applying Tannery's theorem, we may pass to the limit in~\eqref{eq:mod-p-pre-tannery}, obtaining
\begin{equation*}
\lim_{p\to p_0}\Mod_{p,\sigma}(\Gamma_\infty)^{-\frac{1}{p-1}}
= c^{-\frac{1}{p_0-1}}\sum_{k=1}^\infty\left(\frac{\sigma_k|S_k|}{c}\right)^{-\frac{1}{p_0-1}}
= \Mod_{p_0,\sigma}(\Gamma_\infty)^{-\frac{1}{p_0-1}}.
\end{equation*}

Using the positivity of modulus near $p_0$ and the continuity of the logarithm, it follows that
\begin{equation*}
\lim_{p\to p_0}\left(
-\frac{1}{p-1}\log\Mod_{p,\sigma}(\Gamma_\infty)
\right)
= -\frac{1}{p_0-1}\log\Mod_{p_0,\sigma}(\Gamma_\infty),
\end{equation*}
from which it follows that
\begin{equation*}
\lim_{p\to p_0}\Mod_{p,\sigma}(\Gamma_\infty)
= \Mod_{p_0,\sigma}(\Gamma_\infty).
\end{equation*}
\end{proof}

\begin{theorem}
Suppose $\Mod_{p,\sigma}(\Gamma_\infty)>0$ for all $p$ in a neighborhood of $p_0\in(1,\infty)$ and let $e\in E$, then the function $p\mapsto\rho^*_p(e)$, is continuous at $p_0$, where $\rho^*_p$ is the unique optimal density for $p$-modulus.
\end{theorem}

\begin{proof}
Suppose that $e\in S_k$.  The formula in~\eqref{eq:opt-symm-rho} shows that
\begin{equation*}
\rho^*_p(e) = \rho^*_{p,k} :=
\left(
\frac{\sigma_k|S_k|}
{\Mod_{p,\sigma}(\Gamma_\infty)}
\right)^{-\frac{1}{p-1}}.
\end{equation*}
Continuity follows from Theorem~\ref{thm:symm-mod-cont} and the fact that the function $p\mapsto x^{-\frac{1}{p-1}}$ is continuous on $(1,\infty)$ for positive $x$.
\end{proof}
For $\theta >0$ and a straightforward simplification in the formula \eqref{eq:mod-formula-inf-family} shows that the $p$-modulus will be scaled if weights are scaled by some positive numbers.
\begin{lemma}\label{lemma:scaledelta}
Suppose $\Mod_{p, \sigma}(\Gamma_\infty)>0$ for some $1 < p< \infty$ and let $\theta>0$. Then
\begin{equation*}
    \Mod_{p,\theta\sigma}(\Gamma_\infty)=\theta\Mod_{p, \sigma}(\Gamma_\infty)
\end{equation*}
\end{lemma}
In case of $\sigma=\sigma_1+\sigma_2$, then following inequality occurs. 
\begin{lemma}\label{lemma:deltasum}
Let $\Mod_{p,\sigma}(\Gamma_\infty)>0$ for  $1<p<\infty$, and let $\sigma^1,\sigma^2: E \to \mathbb{R}_{>0}$. Then $$\Mod_{p, \sigma^1+\sigma^2}(\Gamma_\infty) \ge \Mod_{p, \sigma_1}(\Gamma_\infty)+\Mod_{p, \sigma_2}(\Gamma_\infty).$$
\end{lemma}
The modulus as function of the edge weights $\sigma$ provides the concavity property. The Lemmas~\eqref{lemma:scaledelta} and~\eqref{lemma:deltasum} provides concavity property of modulus in the infinite tree. 
  
\begin{theorem}\label{theorem:concave}
Let $1 < p< \infty$ and suppose $\Mod_{p, \sigma}(\Gamma_\infty)>0$. The function $\sigma \mapsto \Mod_{p, \sigma}(\Gamma_\infty)$ is concave. 
\end{theorem}

If we bound the edge weights, then we have strong properties on the $p$-modulus. A weight $\sigma\in\mathbb{R}^E_{>0}$ is uniformly elliptic and bounded if there exists constants $\alpha_1>0$ and $\alpha_2>0$ such that $\alpha_1 \le \sigma \le \alpha_2$. The monotonicity property of the functions $x \to x^{-q/p}$ and $x \to x^{-p/q}$ and uniformly elliptic and bounded weights provide the proof of equivalent relation.
\begin{lemma}\label{lem:elliptic-estimate}
Suppose $\sigma$ is a radially symmetric weight that is bounded and uniformly elliptic with $0<\alpha_1\le\sigma\le \alpha_2$.  Then
\begin{equation}\label{eq:equiv_mod}
\alpha_1\Mod_{p,1}(\Gamma_\infty) \le \Mod_{p,\sigma}(\Gamma_\infty) \le \alpha_2\Mod_{p,1}(\Gamma_\infty).
\end{equation}
\end{lemma}
The immediate consequences of the Lemma~\ref{lem:elliptic-estimate} are $p$-modulus on a weighted tree is equivalent to $p$-modulus on the unweighted tree and $1$-modulus is always positive. 

\begin{corollary}\label{cor:sigma-unit-iff}
Suppose $\sigma$ is a radially symmetric weight that is bounded and uniformly elliptic, then $\Mod_{p,\sigma}(\Gamma_\infty)>0$ if and only if $\Mod_{p,1}(\Gamma_\infty)>0$.
\end{corollary}
\begin{corollary}\label{thm:1-mod-positive}
Suppose $\sigma$ is a radially symmetric weight that is bounded and uniformly elliptic, then $\Mod_{1,\sigma}(\Gamma_\infty)>0$.
\end{corollary}
In the case of weights are radially symmetric that is bounded and uniformly elliptic, the Corollary~\ref{cor:infity_modulus} and the Lemma~\ref{lem:elliptic-estimate} implies the following theorem

\begin{theorem}
Suppose $\sigma$ is a radially symmetric weight that is bounded and uniformly elliptic, then 
\begin{equation*}
\lim_{p\to\infty}\Mod_{p,\sigma}(\Gamma_\infty)^{\frac{1}{p}} = 0 = \Mod_{\infty,\sigma}(\Gamma_\infty).
\end{equation*}
\end{theorem}
The $p$-modulus as a function of weights is Lipschitz if weights are bounded elliptic weights in the radially symmetric tree. For this let $\Sigma = \left\{\sigma\in\mathbb{R}^E_{>0}:\alpha_1\le\sigma\le\alpha_2\right\}$ where $\sigma_1>0$ and $\sigma_2>0$. Define a function $F:\sigma\to\mathbb{R}$ by $F(\sigma)=\Mod_{p,\sigma}(\Gamma)$. Then we have following theorem. 
\begin{theorem}
Suppose $\Mod_{p,1}(\Gamma_\infty)>0$ for some $1<p<\infty$. Then the function $F(\sigma)=\Mod_{p,\sigma}(\Gamma)$ is Lipschitz continuous on $\Sigma$ with respect to the $\infty$-norn.
\end{theorem}
\begin{proof}
By Corollary~\ref{cor:sigma-unit-iff} and Theorem~\ref{thm:mod-opt-den} there exist optimal $\rho^*$ and $\hat{\rho}^*$ for $\Mod_{p, \sigma}(\Gamma_\infty)$ and $\Mod_{p, \hat{\sigma}}(\Gamma_\infty)$ respectively. Note that, for any $N\ge 1$,

\begin{align*}
\sum_{k=1}^N \hat{\sigma}_k|S_k|(\rho^*_k)^p
& = \sum_{k=1}^N \sigma_k|S_k|(\rho^*_k)^p
+ \sum_{k=1}^N (\hat{\sigma}_k-\sigma_k)|S_k|(\rho^*_k)^p\\
&\le  \sum_{k=1}^N \sigma_k|S_k|(\rho^*_k)^p
+ \sum_{k=1}^N \frac{|\hat{\sigma}_k-\sigma_k|}{\sigma_k}\sigma_k|S_k|(\rho^*_k)^p\\
& \le  \sum_{k=1}^\infty \sigma_k|S_k|(\rho^*_k)^p + 
\frac{1}{\alpha_1}\sup_k |\hat{\sigma}_k-\sigma_k|\sum_{k=1}^\infty\sigma_k|S_k|(\rho^*_k)^p
\\
&\le \Mod_{p,\sigma}(\Gamma_\infty) +
\frac{1}{\alpha_1}
\Mod_{p,\sigma}(\Gamma_\infty)
\|\hat{\sigma}-\sigma\|_\infty
\\
&\le \Mod_{p,\sigma}(\Gamma_\infty) +
\frac{\alpha_2}{\alpha_1}
\Mod_{p,1}(\Gamma_\infty)
\|\hat{\sigma}-\sigma\|_\infty.
\end{align*}
Taking the limit as $N\to\infty$ shows that
\begin{equation*}
     \Mod_{p, \hat{\sigma}}(\Gamma_\infty)-  \Mod_{p, \sigma}(\Gamma_\infty) \le \frac{\alpha_2}{\alpha_1}\Mod_{p,1}(\Gamma_\infty)\|\hat{\sigma}-\sigma\|_\infty.
\end{equation*}
Repeating the same argument with $\sigma$ and $\hat{\sigma}$ interchanged establishes the theorem.
\end{proof}

\section{Critical value of $p$ on radially symmetric infinite binary tree}\label{Chapter:critical_exponent}

In this section, we focus on $1-2$ radially symmetric trees and investigate how the modulus and size of trees are related. We aim to determine whether a positive or zero modulus indicates a bushy or skinny tree, respectively. In addition, we explore the existence of a critical exponent for these trees. Consider $\sigma \equiv 1$, the formula for $p$-modulus of the family of descending paths given in~\eqref{eq:mod-formula-inf-family} can be expressed as
\begin{equation}\label{eq:mod-formula-unweighted}
  \Mod_{p,1}(\Gamma_{\infty})= \left(\sum_{k=1}^\infty|S_k|^{-\frac{q}{p}}\right)^{-\frac{p}{q}}.
\end{equation}
In particular, when $C(e) \equiv 1$, the $\Mod_{p,1}(\Gamma_{\infty})=0$ for any $p \in (1, \infty)$ and when $C(e) \equiv 2$, the

\begin{equation*}
    \Mod_{p,1}(\Gamma_\infty) =
    \left(\sum_{k=1}^\infty(2^k)^{-\frac{q}{p}}\right)^{-\frac{p}{q}}
    = 2\left(1-2^{-\frac{q}{p}}\right)^{\frac{p}{q}},
\end{equation*}
which is positive for any $p\in(1,\infty)$. This shows that $p$-modulus is always positive if each shell has at least two children and zero if each shell has one child or the tree is sparse. For any $1$-$2$ radially symmetric trees, the $p$-modulus is either positive or zero depending on the parameter $p$. From the Lemma~\ref{lem:cut-inequality}, the $1$-modulus of the family of descending paths in $1$-$2$ radially symmetric tree is positive which is
\begin{equation}\label{eq:1-modulus}
    \Mod_{1,1}(\Gamma_\infty) =|S_1|.
\end{equation}
 From the Corollary~\ref{eq:infinitymodulus_as_limit}, the $p$-modulus decays to zero as $p$ approaches to infinity, that is, 
 \begin{equation}\label{eq:infity_modulus}
    \lim_{p \to \infty}\Mod_{p,1}(\Gamma_\infty)^{\frac{1}{p}}=0.
\end{equation}
Finally, the monotonicity property of $p$-modulus for a $1$-$2$ radially symmetric trees from the Lemma~\ref{lemma:symmetric_pmono} is 
\begin{equation}\label{eq:monotone}
    \Mod_{p', 1}(\Gamma_\infty)\le \Mod_{p, 1}(\Gamma_\infty), \qquad p \le p'
\end{equation}
The equations~\eqref{eq:1-modulus}, \eqref{eq:infity_modulus}, and \eqref{eq:monotone} imply that there may be a value for the parameter $p$ such that $p$-modulus becomes zero for all larger values of $p$. Therefore, we define the critical parameter or critical exponent for the $p$-modulus in a $1$-$2$ radially symmetric tree as 
\begin{equation*}
    p_c=\sup\{p >1 : \Mod_{p,1}(\Gamma_{\infty})>0\}
\end{equation*}
with the understanding that $p_c$ may be either $1$ (if $\Mod_{p,1}(\Gamma_\infty)=0$ for all $p>1$) or $+\infty$ (if $\Mod_{p,1}(\Gamma_\infty)>0$ for all $p>1$).
The critical exponent measures of how ``bushy'' or ``skinny'' the tree is. At one extreme we have the $1$-ary tree for which $p_c=1$.  On the other hand, we have the binary tree for which $p_c=+\infty$. More generally, all radially symmetric trees for which each vertex has at least two children have $p_c= + \infty$. In this sense, the critical exponent can be considered as a measure of dimension for very sparse infinite trees. 

A radially symmetric binary tree can be represented by a sequence $\{a_i: a \in \{1,2\}, i= 1, 2, 3, \cdots\}$ where $a_i$ represents the children atthe generation $i$. Let $K=\{k_0, k_1, k_2, \cdots\}$ be the sequence of indices where $a_i=2$. Define a \emph{skip sequence} of indices $c_i=k_i-k_{i-1}$ for $i=1, 2, 3, \cdots$, then $p$-modulus in \eqref{eq:mod-formula-unweighted} takes the form
\begin{equation}\label{mod:1-2}
    \Mod_{p,1}(\Gamma_{\infty})=\left[\sum_{k=1}^{\infty}\frac{c_k}{2^{\frac{k}{p-1}}}\right]^{1-p} \qquad \text{for} \quad k_0=1
\end{equation} 
\begin{equation}\label{mod:firstone tree}
    \Mod_{p,1}(\Gamma_{\infty})=\left[m+\sum_{k=1}^{\infty}\frac{c_k}{2^{\frac{k}{p-1}}}\right]^{1-p} \qquad \text{for} \quad k_0=m>1
\end{equation}
The terms $c_i$ of skip sequence are the gap to generations which has two children in a $1$-$2$ radially symmetric tree. When $c_i=1$ for all $i$, tree becomes a binary tree and in this case, the $p$-modulus is finite since the series in the right side of~\eqref{mod:1-2} converges for all $p >1$. When $c_i=\infty$, the tree becomes a $1$-ray tree and in this case $p$-modulus is zero by ~\eqref{mod:1-2} for all $p >1$. This shows that $p$-modulus of family of descending paths in $1$-$2$ tree either zero or finite depending on the parameter $p$ and how the number of children grows in each generation. The Theorem~\ref{theorem:existence_critical} ensures a radially symmetric binary tree trees can have nontrivial critical $p_c$ (that is, $1<p_c<\infty$).
\begin{proof}[Proof of Theorem~\ref{theorem:existence_critical}]
For $k=1,2,\ldots$, define
\begin{equation*}
c_k = \left\lceil2^\frac{k}{r-1}\right\rceil.
\end{equation*}
Since each $c_k\ge 1$, the sequence $\{c_k\}$ is a skip sequence for some 1-2 tree.  Moreover,
\begin{equation*}
\left(2^{\frac{1}{r-1}-\frac{1}{p-1}}\right)^k \le
\frac{c_k}{2^{\frac{k}{p-1}}} \le
\left(2^{\frac{1}{r-1}-\frac{1}{p-1}}\right)^k
 + \left(2^{-\frac{1}{p-1}}\right)^k
\qquad\text{for all $k$}.
\end{equation*}
Comparison with the geometric series shows that the infinite series converges if and only if $1<p<r$.
\end{proof}

The \emph{$2$-modulus} of the family of descending paths in an infinite tree is effective conductance and related to transience or recurrence of a random walk (see Corollary~\ref{corollary:transient}). In particular, a random walk is recurrent (or transient) in relatively \emph{skinny} (or \emph{bushy}) unweighted radially symmetric trees. The critical exponent $p_c$ tells us about a random walk is transient or recurrent in $1$-$2$ in tree. 

\begin{theorem}\label{theorem:random_1-2tree}
Let $p_c$ be the critical exponent for $\Mod_{p,1}(\Gamma_{\infty})$. Then
\begin{enumerate}
    \item If $p_c >2$ then the random walk is transient. 
    \item If $p_c < 2$ then the random walk is recurrent.
\end{enumerate}
\end{theorem}
\begin{proof}
For 1), suppose the random walk is recurrent, then by the Corollary~\ref{corollary:transient},  $\Mod_{2,1}(\Gamma_\infty)=0$, so $p_c\le 2$.  For 2), suppose the random walk is transient, then by the Corollary~\ref{corollary:transient} $\Mod_{2,1}(\Gamma_\infty)>0$, so $p_c\ge 2$.
\end{proof}

Note that the case $p_c=2$ cannot be decided in this way since there is currently no general theory to determine whether $\Mod_{p_c,1}(\Gamma_\infty)$ is positive or zero. As an example, take the skip sequence $c_k=2^k$ for all $k$, then by the Theorem~\ref{theorem:existence_critical} the critical exponent is $p_c=2$ and the $\Mod_{2,1}(\Gamma_\infty)=0$. On other hand pick a skip sequence as 
\begin{equation*}
c_k=\left\lceil\frac{2^k}{k^2}\right\rceil\ge 1\quad\text{for }k=1,2,3,\ldots.
\end{equation*}
Then 
\begin{equation*}
    \sum_{k=1}^\infty \frac{c_k}{2^{\frac{k}{p-1}}}=\sum_{k=1}^\infty\frac{\left\lceil\frac{2^k}{k^2}\right\rceil}{2^{\frac{k}{p-1}}} \ge \sum_{k=1}^\infty\frac{\frac{2^k}{k^2}}{2^{\frac{k}{p-1}}} =\sum_{k=1}^\infty\frac{2^{\left(\frac{p-2}{p-1}\right)k}}{k^2}
\end{equation*}
By the ratio test, the series in the right side diverges if $p>2$
\begin{equation*}
   \lim_{k \to \infty} \frac{2^{\left(\frac{p-2}{p-1}\right)(k+1)}}{(k+1)^2} \cdot \frac{k^2}{2^{\left(\frac{p-2}{p-1}\right)k}}=2^{\frac{p-2}{p-1}}
\end{equation*}
This shows that $\Mod_{p,1}(\Gamma_\infty)=0$ for $p>2$.
Therefore, the critical exponent satisfies $p_c\le 2$ but at $p=2$, we have
\begin{equation*}
    \Mod_{2,1}(\Gamma_\infty)^{-1}=\sum_{k=1}^\infty \frac{c_k}{2^k} =\sum_{k=1}^\infty\frac{\left\lceil\frac{2^k}{k^2}\right\rceil}{2^k} \le  \sum_{k=1}^\infty\frac{\frac{2^{k}}{k^2}+1}{2^k}=\sum_{k=1}^\infty\frac{1}{k^2} + \sum_{k=1}^\infty\frac{1}{2^k}.
\end{equation*}
Both series on the right converge absolutely so by comparison test, we have  $\Mod_{2,1}(\Gamma_\infty)>0$ and, therefore $p_c=2$.

Next, in a weighted infinite rooted $1$-$2$ radially symmetric tree the positive (or zero) \emph{$p$-modulus} may not imply a tree is dense (or thin) in the sense of the number of children. Consider the path tree that is, $S_k=1$ for all $k$, then the formula for \emph{$p$-modulus} is expressed in 
\begin{equation}\label{for:1-2modweighted}
    \Mod_{p,\sigma}(\Gamma_\infty)=\left(\sum_{k=1}^\infty\sigma_k^{-\frac{q}{p}}\right)^{-\frac{p}{q}}
\end{equation}

In this case, given $p>1$, the weight $\sigma$ determines whether an infinite tree has the positive or zero modulus. For an example, let $\sigma_k = 2^k$, then the series on the right hand side of~\eqref{for:1-2modweighted} for $1<p < 2$ converges and diverges for any $p\ge 2$. Hence $\Mod_{p,\sigma}(\Gamma_\infty)>0$ is positive for $1<p < 2$ and $\Mod_{p,\sigma}(\Gamma_\infty)>0$ is zero for $p \ge 2$. Consider another example of a dense infinite tree, $|S_k|=2^k$ for $k=1, 2, \cdots$ then the \emph{$p$-modulus} is given by 
\begin{equation}\label{for:1-2modweighted2}
    \Mod_{p,\sigma}(\Gamma_\infty)=\left(\sum_{k=1}^\infty(\sigma_k2^k)^{-\frac{q}{p}}\right)^{-\frac{p}{q}}
\end{equation}
If $\sigma_k=2^{-k}$ for $k= 1, 2, \cdots$, then  the \emph{$p$-modulus} zero for any $ p >1$. This shows that \emph{$p$-modulus} on a dense $1$-$2$ infinite tree has modulus zero as the weight decay proportional to the number of children of each shell. Hence, we define the critical exponent for weighted $1$-$2$ radially symmetric tree.

\begin{equation*}
    p_c^\sigma=\sup\{p>1: \Mod_{p, \sigma}(\Gamma_\infty) > 0\}
\end{equation*}

In the case of $\sigma \equiv 1$, we have the critical exponent $p_c$ for the unweighted $1$-$2$ radially symmetric tree. The following theorem characterize the nature of a random walk in the weighted  $1$-$2$ radially symmetric tree. The proof is similar arguments of the Theorem~\ref{theorem:random_1-2tree}.
\begin{theorem}\label{theorem:random_1-2treeweighted}
Let $p_c$ be the critical exponent for $\Mod_{p,\sigma}(\Gamma_{\infty})$. Then
\begin{enumerate}
    \item If $p^\sigma_c >2$ then the random walk is transient. 
    \item If $p^\sigma_c  < 2$ then the random walk is recurrent.
\end{enumerate}
\end{theorem}
In the case, the weight is bounded and uniformly elliptic, then it is more easier to get critical exponent of $1$-$2$ radially symmetric infinite tree. Recall the relation~\eqref{eq:equiv_mod} where the weights are bounded and uniformly elliptic the \emph{$p$-moduli} $\Mod_{p,\sigma}(\Gamma_\infty)$ and $\Mod_{p,1}(\Gamma_\infty)$ are equivalent. That is, 
\begin{equation*}
    \alpha_1\Mod_{p,1}(\Gamma_\infty) \le \Mod_{p,\sigma}(\Gamma_\infty) \le \alpha_2\Mod_{p,1}(\Gamma_\infty)
\end{equation*}

\begin{theorem}\label{the:equalcriticalvalue}
For $0<p<\infty$ and $\sigma$ be radially symmetric, bounded, and uniformly elliptic. Then $$p_c^\sigma=p_c$$.
\end{theorem}
\begin{proof}
If $\Mod_{p,1}(\Gamma_\infty)=0$ then $\Mod_{p,\sigma}(\Gamma_\infty)=0$. Hence $p_c^\sigma= p_c=1$.  Let $\Mod_{p,1}(\Gamma_\infty)>0$, then  by first inequality we have $p_c\le p_c^\sigma$ and by the second inequality, we have, $p_c^\sigma\le p_c$.
\end{proof}
The following corollary is a direct consequence of the Theorem~\ref{the:equalcriticalvalue} and it relates the  critical value of parameter $p$ with  a random walk in unweighted and weighted infinite tree. 
\begin{corollary}
Let $\sigma$ be a bounded and uniformly elliptic weight and let $G=(G,V
,\sigma,o)$ be a $1$-$2$ radially symmetric tree. Then a random walk is transient (recurrent) if and only if $p^{\sigma}_c>2$ ($p^{\sigma}_c<2$).
\end{corollary}

\section{Conclusion}
In this paper, we have extended the concept of $p$-modulus on finite graphs to proper infinite graphs. The $p$-modulus for the descending paths is expressed as the limit of the $p$-modulus of truncated trees. In particular, we derived the formula~\eqref{eq:mod-formula-inf-family} for computing the p-modulus in the radially symmetric tree, which converges if and only if the modulus is positive. A special dual formulation provides a lower bound of the p-modulus. However, it remains an open question to derive a standard Lagrangian of the $p$-modulus in the appropriate infinite tree. We focus on a special class of $1$-$2$ radially symmetric infinite tree and the existence of a critical parameter $p_c$ that separates zero and positive $p$-modulus. The critical value is also related to the transience or recurrence of a random walk. Does such a critical exponent exist for a general tree modulus?  Furthermore, in general (non-symmetric) trees, it is an open question to know the existence of the optimal density and it is the limit of the optimal densities for the $p$-modulus in the truncated trees as long as the limiting density is nowhere zero. If the limit exists, does it converge point-wise or uniformly? Is the limit density unique? 

\bibliographystyle{plain}
\bibliography{references} % see references.bib for bibliography management
\end{document}